\documentclass[10pt,reqno]{amsart}
\usepackage{amsmath}
\usepackage{amssymb}
\usepackage{amsthm}
\usepackage[usenames]{color}
\usepackage{eepic,epic}

\newtheorem{thm}{Theorem}[section]

\newtheorem{lem}[thm]{Lemma}

\theoremstyle{definition}

\theoremstyle{remark}

\numberwithin{equation}{section}

\newcommand{\dist}{\text{dist }}

\begin{document}

\title[]
{Maximal functions for multipliers on Homogeneous groups}

\author{Woocheol Choi}
\subjclass[2000]{Primary}

\address{School of Mathematical Sciences, Seoul National University, Seoul 151-747, Korea}
\email{chwc1987@snu.ac.kr}
\maketitle

\begin{abstract}
We study  $L^p$ boundedness of maximal multipliers on stratified groups.
\end{abstract}

\section{introduction}
\


We consider  a stratified group $G$  with homogeneous dimension $Q$ and  set ${L}$ be a left invariant sub-laplacian on $G$. Denote by $\{ E(\lambda) : \lambda \geq 0\}$ the spectral resolution of ${L}$. Then, for a bounded function $m : [0,\infty) \rightarrow \mathbb{R}$, we can define the multiplier operator $m({L})$ as
\begin{eqnarray*}
m(L)  = \int_0^{\infty} m(\lambda) d E(\lambda).
\end{eqnarray*}
The sufficient conditions on a function $m$ which guarantee $\| m(L) f \|_p \lesssim \| f\|_p$ for all $f \in S(G)$ has been investigated widely in the literature (see,  e.g, Folland-Stein \cite{FS2}). A sharp result was obtained by Christ \cite{C1} and Mauceri-Meda \cite{MM} independently. They proved that the following condition
\begin{eqnarray*}
\sup_{\lambda \in \mathbb{R}^{+}} \| \phi(s) m (\lambda s) \|_{L_2^{\alpha}} < \infty  \quad\qquad \textrm{for some}~\alpha > \frac{Q}{2}
\end{eqnarray*}
is a  sufficient condition. Next, we consider the maximal function $\mathcal{M}_m f (x) := \sup_{t >0} | m ( t L) f (x)|$. Mauceri-Meda (see \cite[Theorem 2.6]{MM}) proved that $\mathcal{M}_m $ is bounded on $L^p (G)$ if the condition
\begin{eqnarray}\label{kzpm2}
\sum_{k \in \mathbb{Z}} \| \phi(\cdot) m(2^k \cdot) \|_{H_2^{s}} = D < \infty
\end{eqnarray}
holds for some $s > Q(\frac{1}{p} -\frac{1}{2}) + \frac{1}{2}$ when $p\in (1,2]$, or for some $ s > (Q-1) (\frac{1}{2}-\frac{1}{p}) +\frac{1}{2}$ when $p \in [2,\infty]$. However, the summability condition \eqref{kzpm2} would be sharpened as we shall see in this paper. For this we shall prove the result on maximal multipliers which was considered on Euclidean space by Christ-Grafakos-Honzik-Seeger \cite{CGHS} and Grafakos-Honzik=Seeger \cite{GHS}.
\begin{thm}\label{thm1} For $1 \leq r < 2$, suppose that there is $\alpha >Q/r$ such that $m_1, \dots, m_N$ satisfy the condition 
\begin{eqnarray*}
\sup_{\lambda \in \mathbb{R}} \| \phi(s) m_i (\lambda)\|_{L^{\alpha}_2} \leq B \qquad \textrm{for}  ~ i = 1,\cdots,N.
\end{eqnarray*}
Then, for all $p \in (r,\infty)$, we have 
\begin{eqnarray*}
\| \sup_{i=1,\dots,N} | m_i (L) f | \|_p \leq C_{p,r} B \sqrt{\log(N+1)} \| f\|_p.
\end{eqnarray*}
\end{thm}
Combining this theorem with an argument in \cite{GHS} we shall prove the following theorem on the maximal operator $\mathcal{M}_m$.
\begin{thm}\label{thm2}
Suppose that 
\begin{eqnarray*}
\| \phi(\cdot) m(2^k \cdot) \|_{H^{\alpha}_{2}} \leq \omega (k), \quad k \in \mathbb{Z},
\end{eqnarray*}
holds for some $\alpha $ and suppose that the nonincreasing rearrangement $\omega^*$ satisfies 
\begin{eqnarray*}
\omega^{*}(0) + \sum_{l=1}^{\infty} \frac{\omega^{*}(l)}{l \sqrt{\log l}} < \infty.
\end{eqnarray*}
If $\alpha > Q/r + 1 $ for some $1 \leq r <2$, then  $ \mathcal{M}_m$ is bounded on $L^p (G)$ for $p \in (r, \infty)$.
\end{thm}
We also study  the maximal function of multipliers given in Theorem \ref{thm1} on product spaces of stratified groups. As an application, we shall obtain a boundedness  property for the maximal function of joint spectral multipliers on the Heisenberg group. Let $G$ be a product space of $n$ stratified groups $G_1, \cdots, G_n$. Consider sublaplcian $L_j$, $1\leq j \leq n$ and its lifting to $G$ denoted by $L_j^{\sharp}$. Under the following assumption on $m$:
\begin{eqnarray}\label{5A}
| (\xi_1 \partial_{\xi_1})^{\alpha_1} \cdots (\xi_n \partial_{\xi_n})^{\alpha_n} m (\xi_1,\cdots,\xi_n)| \leq C_{\alpha}
\end{eqnarray}
for all $\alpha_j \leq N$, with $N$ large enough, it was proved in \cite{MRS} that the multiplier $m(L_1^{\sharp},\cdots, L_s^{\sharp})$ is bounded on $L^p (G)$.
For these multipliers, we shall derive the following result on the maximal function. 
 \begin{thm}\label{thm3}
Suppose that functions $m_1, \dots, m_N$ on $(\mathbb{R}_{+})^{n}$ satisfy the condition \eqref{5A} uniformly. Then, for all $p \in (1,\infty)$, we have the ineqaulity
\begin{eqnarray*}
\left\| \sup_{1 \leq i \leq N} | m_i (L_1^{\sharp}, \cdots, L_s^{\sharp}) f |\right\|_p \leq C(G) ( \log (N+1))^{s/2} \left\| f \right\|_{p} (G).
\end{eqnarray*}
\end{thm} 

\

As an application of this theorem we may obtain similar result for the joint spectral multipliers on the Heisenberg group $\mathbb{H}_n$ with the sub-Laplacian $\Delta$ and $T = \frac{\partial}{\partial t}$. The $L^p$ boundedness of the joint spectral multiplier $m(\Delta, iT)$ was studied in \cite{MRS}. Using Theorem \ref{thm3} and the transference method of \cite{CW}, we shall prove the following theorem.
\begin{thm}\label{thm4}
Suppose that functions $m_1, \dots, m_N$ on $(\mathbb{R}_{+})^{2}$ satisfy the condition \eqref{5A} uniformly. Then, for all $p \in (1,\infty)$, we have
\begin{eqnarray*}
\left\| \sup_{1 \leq i \leq N} | m_i (\Delta, iT) f |\right\|_p \leq C(\mathbb{H}_n) ( \log (N+1)) \left\| f\right\|_{p}.
\end{eqnarray*}
\end{thm}

In order to prove Theorem 1.1 we shall use the noble argument of Grafakos-Honzik-Seeger \cite{GHS} which use the good $\lambda$ inequality for martingale operators proved by Chang \emph{et al.}\cite{CWW}. For using the martingales, they proved several lemmas which show the cancellation property arising when we compose the $2^{-k}$ martingale operators and the Littlewood-Paley projection  $P_j $. To show it, one can use a simple technique on the frequency side to localize the support of the kernel of $P_j $ to the area $\{ |x| : |x| \leq 2^{-j +3}\}$. In order to use this argument in our setting, we first need an analogue for stratified groups of the dyadic martingales on Euclidean space. Fortunately we can use the martingales on  homogeneous space constructed by Christ \cite{C2}. However, due to the technical difficulties of the Fourier transform on stratified groups, it is not easy to adapt the former technique to localize the Littlewood-Paley projections on stratified groups to prove the cancellation property arising in composing the martingale operators and the  projections. Instead, we shall obtain weaker versions of the lemmas through a more direct estimate after a suitable partitioning of the kernel of the projection. 
\

On stratified groups, it is not easy as on Euclidean space to obtain information on kernels of multipliers. Nevertheless, a technique was developed by Stein and Folland \cite{FS2} using the kernel of the heat semi-group $e^{-t L}, t >0$. In addition, Christ \cite{C1} and Mauceri-Meda \cite{MM} obtained a sharp information on the integrability of the kernels by using Plancherel formula on stratified groups (see Lemma \ref{MME2} below).  We shall extend their lemma to $q > 2$ with using the heat kernel to bound multiplier operators with localized multiplier function pointwisely by Hardy-Littlewood maximal function (see Lemma \ref{tkfx}). 
\

For proving Theorem \ref{thm3} we shall use an inductional argument of Honzik \cite{H} who made use of the intermediate square functions (see Section 4) to study the boundedness of maximal functions of marcinkiewicz multipliers. Note that this multiplier corresponds to multiplier on the product space $G = \mathbb{R} \times \cdots \times \mathbb{R}$. In \cite{H} the proof use the $L^p$ boundedness property of the multipliers with characteristic functions on rectangles in $\mathbb{R}^n$. Therefore, we cannot adapt his argument directly to prove our theorem. Remark also that we shall make a modification of a main lemma \cite[Lemma 2.2]{P}. Precisely, we shall prove the lemma for $g = (1- \mathcal{E}_m) f$ instead of $g = E_{N,\cdots,N} f$ as in \cite{H} (see Lemma \ref{for2m}).
\

This paper is organized as follows. In section 2, we study properties of kernels corresponding to multiplier operator and review on the martingale on homogeneous space. In section 3, we prove several lemmas explaining cancellation between martingale and Littlewood-Paley operators. Based on these lemmas, we prove our main theorems in the last section. In section 4, we generalize the above criterion to product spaces and we obtain boundedness for maximal functions of the joint spectral multipliers.
\

We shall use the notation $A \lesssim B$ to indicate an inequality $A \leq CB$ where $C$ may only depend on the background spaces and the index $p$ of the space $L^p$ used in the inequality.

\section{ Kernels of multipliers on Stratified groups}
In this section we recall the background of stratified groups and two lemmas about kernels of multipliers on the spaces from Folland-Stein \cite{FS} and Mauceri-Meda \cite{MM}. Then, we shall prove two lemmas to bound a multiplier with a localized spectrum by Hardy-Littlewood maximal functions. 
\

Let $\mathfrak{g}$ be a finite-dimensional nilpotent Lie algebra of the form
\begin{eqnarray*}
\mathfrak{g} = \bigoplus_{i=1}^{s} \mathfrak{g}_i
\end{eqnarray*}
where $[\mathfrak{g}_i, \mathfrak{g}_j ] \subset \mathfrak{g}_{i+j}$ for all $i, j $. Set $G$ be the associated simply connedted Lie group. Then, its homogeneous dimension is $ Q = \sum_j j \cdot \textrm{dim}(\mathfrak{g_j})$. We call it a stratified group when $\mathfrak{g_1}$ generates $\mathfrak{g}$ as a Lie algebra. We shall always let $G$ be a stratified group in this paper. 
\

We denote by $\{ \delta_r : r >0 \}$ a family of dilations of the Lie algebra $\mathfrak{g}$ which satisfy $\delta_r X = r^j X$ for $X \in \mathfrak{g}_j$, and is extended by linearity. We shall also denote by $\{ \delta_r : r> 0\}$ the induced family of dilations of $G$. They are group automorphisms. We define a homogeneous norm of $G$ to be a continuous function $| \cdot | : G \longrightarrow [0,\infty)$ which is, $C^{\infty}$ away from $0$, and satisfies $|x|=0 \Leftrightarrow x=0$ and $| \delta_r x | = r |x|$ for all $ r \in \mathbb{R}^{+}$, $x \in G$.
\

We denote $S(G)$ be the space of Schwartz functions in $G$. Now we choose any finite subset $\{X_k\}$ of $\mathfrak{g}_1$ which spans $\mathfrak{g}_1$. We may identify each $X_k$ with a unique left-invariant vector field on $G$. We also denote it by $X_k$. Then we define a sublaplacian as $L = - \sum X_k^2$, which is a left-invariant second-order differential operator. $L^p (G)$ is defined with respect to a bi-invariant Haar measure. As an operator on $\{ f \in L^2 (G) : L f \in L^2 (G)\}$, $L$ is self-adjoint. Therefore it admits a spectral resolution $L = \int_0^{\infty} \lambda dP_{\lambda}$. For a bounded Borel function $m$ on $[0,\infty)$, we define the bounded operator $m(L)$ on $L^2$ by 
\begin{eqnarray*}
m(L) = \int^{\infty}_0 m (\lambda) dP_{\lambda}.
\end{eqnarray*}
By the Schwartz kernel theorem, there exists a tempered distribution $k_m$ on $G$ satisfying $m(L) f = f * k_m$ for all functions in $G$. For a tempered distribution $k$ on $G$, we denote by $k_{(t)}$ for $t >0$ the distribution satisfying
\begin{eqnarray*}
\langle k_{(t)}, f \rangle = \langle k, f \circ \delta_t \rangle
\end{eqnarray*}
for all $ f \in S$. If $ k$ is a measurable function on $G$, then $k_{(t)} (x) = \frac{1}{t^Q} f( \frac{x}{t})$. 
\

The heat semigroup $\{ e^{-t L}\}_{t >0}$ on $G$ can be defined as 
\begin{eqnarray*}
e^{-t L} = \int^{\infty}_0 e^{-\lambda t } dP_\lambda
\end{eqnarray*}
and we set  $h_t (x)$ be the heat kernel satisfying $e^{-t L} f = f * h_t $ for all $f \in L^2$. Let us simply write $h(x)$ for $h_1 (x)$. Then we have $h_t (x) = h_{(\sqrt{t})} (x)$ and it was proved in \cite{JS} that $|h(x)| \leq C e^{-c_0 |x|^2}$ for some $ c_0, C \in \mathbb{R}^{+}$.
\
The next lemma is from \cite[Lemma 6.29]{FS2}.

\begin{lem}\label{dilation} If M is a bounded Borel function on $(0,\infty)$, let $K$ be the distribution kernel of $M(L)$. Then for any $t>0$, if $M_{(t)}(\lambda) = M(t\lambda)$, the distribution kernel of $M_{(t)}(L)$ is $K_{\sqrt{t}}$
\end{lem}
We recall \cite[Lemma 1.2]{MM} which was essential to prove the sharp multiplier theorem on stratified groups.
\begin{lem}\label{MME2}
For $\alpha \geq 0$, $1 \leq p \leq2$, we suppose that $s > \alpha/p + Q(1/p -1/2)$. Then, for a function $m \in H^{s}_2(\mathbb{R}_{+})$ with its support in $(1/2,2)$ and the distribution kernel $k$ of $m(L)$, there exists a constant $C_I >0$  for each multi-index $I$ such that
\begin{eqnarray}\label{MME}
\int_G |x|^{\alpha} |Y^{I}k(x)|^p dx \leq C_I \| m\|_{H^s_2}^p.
\end{eqnarray}
\end{lem}
 We shall use this lemma with $p=2$, that is, 
\begin{eqnarray}\label{m2}
\int_G |x|^{\alpha} |Y^{I}k(x)|^2 dx \leq C \| m \|_{H_2^{s}}, \quad s> \frac{\alpha}{2}.
\end{eqnarray}
However, in order to bound a multiplier pointwisely by Hardy-Littlewood maximal function, we derive a version of \eqref{MME} with  $p> 2$.
\begin{lem}\label{HigherKernel} Suppose that $m$ is a function in $H_2^{s}(\mathbb{R}_+)$ supported in $(1/2,2)$ with $ s> \alpha /2$. Let $k$ be the distribution kernel of $m(L)$. Then, for any index $I$ we have
\begin{eqnarray}\label{minfty}
\sup_{x \in G} (1+|x|)^{\frac{\alpha}{2}} |X^{I} k(x)| \lesssim \| m\|_{H_2^{s}}.
\end{eqnarray}
and 
\begin{eqnarray}\label{ml2}
\int_G |x|^{\alpha \frac{q}{2}} |Y^{I} k(x)|^{q} dx \lesssim \|m\|_{H_2^{s}}^{q}.
\end{eqnarray}
for each $q >2$.
\end{lem}
\begin{proof}
Set $m_1 (\lambda) = e^{\lambda} m(\lambda)$ and  $K_1$ be the distribution kernel of $m_1(L)$. Because the support of $m$ is contained in $(1/2,2)$, the $H^{s}_2$ norms of $M$ and $M_1$ are comparable. Since $m(L) = e^{-L} m_1 (L) = m_1(L) e^{-L}$ we have $K = h * K_1 = K_1 * h $. Hence $K$ is $C^{\infty}$ and 
\begin{eqnarray}\label{xik}
X^{I} K = K_1 * X^{I} h, \qquad Y^{I} K = Y^{I} h * K_1.
\end{eqnarray}
Since $h \in S(G)$ we have $\| |x|^{N} X^{I} h(x) \|_{L^2} \lesssim 1$ for any $N>0$. Thus using \eqref{m2}, \eqref{xik} and a triangle inequaltiy we get,
\begin{eqnarray*}
(1 + |x|)^{\alpha/2} |X^{I} K(x)| &\lesssim&  \int (1 + |y|)^{\alpha/2} |K_1 (y) | (1+ |y^{-1}x|)^{\alpha/2} |X^{I} h (y^{-1}x,1)| dy
\\
&\lesssim & \left(\int (1+|y|)^{\alpha} |K_1 (y)|^{2} dy \right)^{1/2}\left( \| |x|^{\alpha/2} X^{I} h(x) \|_{L^2}\right)
\\
&\lesssim & \| m\|_{H_2^{s}}.
\end{eqnarray*}
It proves \eqref{minfty}. For $q>2$, using this bound with \eqref{m2} and H\"older's inequality, we obtain\begin{eqnarray*}
\int_G |x|^{\alpha \frac{q}{2}} |Y^I k(x)|^q dx &\lesssim& \sup_{x\in G} |x|^{\alpha \frac{q-2}{2}} |Y^I k(x)|^{q-2} \int_G |x|^{\alpha} |Y^{I} k(x)|^2 dx .
\\
&\lesssim& \|m\|_{H_2^{s}}^{q-2} \|m\|_{H_2^{s}}^2 = \| m \|_{H_2^{s}}^{q}.
\end{eqnarray*}
This yields the asserted bound \eqref{ml2}.
\end{proof}
 In order to split the spectrum of a multiplier dyadically, we choose a bump function $\phi \in C^{\infty}$ supported on $[\frac{1}{2}, 2]$ satisfying $\sum_{j \in \mathbb{Z}} \phi (2^j \xi) =1 $ for all $ \xi \in \mathbb{R}^{+}$. Then, letting $m_j (\xi) = \phi (2^{-j} \xi) m(\xi)$ we have
\begin{eqnarray*}
m(L) = \sum_{j \in \mathbb{Z}} m_j (L).
\end{eqnarray*}
We set  $\tilde{m}_k (s) := m_k (2^k s)$ and the maximal function of higher order, $M_r f(x) = \left(M(|f|^r)(x)\right)^{1/r}$ for $r>1$. Then we have the following lemma.
\begin{lem}\label{tkfx}
$| m_k (L) f (x) | \lesssim M_r f (x) \cdot \| \tilde{m}_k \|_{H^{s}}, \quad s > Q/r$, \quad $ r \leq 2$.
\end{lem}
\begin{proof} Denote by $K_k$ (resp. $\tilde{K}_k$) the kernel of the operator $m_k (L)$(resp. $\tilde{m}_k (L)$). It follows from Lemma \ref{dilation} that 
\begin{eqnarray*}
\tilde{K}_k (x) = (K_k)_{\sqrt{2^k}}(x) = 2^{-kQ/2} K_k (\frac{x}{2^{k/2}}).
\end{eqnarray*}
Since $r'>2$ we get from Lemma \ref{HigherKernel} that
\begin{eqnarray}\label{partition}
\int_G |x|^{\alpha \frac{r'}{2}} |\tilde{K}_k (x)|^{r'} dx \lesssim \| \tilde{m}_k\|_{H^{s}}^{r'}, \quad \textrm{for all} ~0\leq \alpha < 2s.
\end{eqnarray}  
Set $\tilde{K}_{k,l} (x) = \tilde{K}_k (x) \cdot 1_{\{2^{l-1} \leq |x| < 2^l\}}$ for $l \in \mathbb{N}$ and $\tilde{K}_{k,0}(x) = \tilde{K}_k (x) \cdot 1_{\{ |x| < 1\}}$. Then, it follows from \eqref{partition} that 
\begin{eqnarray}\label{alpha}
\sup_{l\geq 0} 2^{l \alpha \frac{r'}{2}} \int |\tilde{K}_{k,l}(x)|^{r'} dx \lesssim \| \tilde{m}_k\|_{H^{s}}^{r'} \quad \textrm{for} \quad 0 \leq \alpha < 2s.
\end{eqnarray} 
Since $\frac{2Q}{r} <2s$ we can choose a small $\epsilon >0$ and take $\alpha$ so that $\alpha =\frac{2Q}{r}+\epsilon < 2s$. We then deduce the following estimate
\begin{eqnarray*}
|m_k (L) f (x)| &=&  \left| \int_G 2^{kQ/2} \tilde{K}_k (2^{k/2} y) f(x y^{-1}) dy\right|
\\
 &=&\left| \sum_{l=0}^{\infty} \int_G 2^{kQ/2} \tilde{K}_{k,l} (2^{k/2} y) f (x y^{-1}) dy \right|
\\
&\lesssim & \sum_{l=0}^{\infty} \left( \int_G 2^{kQ/2} |\tilde{K}_{k,l}(2^{k/2} (y))|^{r'} dy\right)^{1/r'} \left(2^{kQ/2}\int_{|y| \leq 2^{l -k/2}} |f(x y^{-1})|^{r} dy\right)^{1/r}
\\
&\lesssim& \sum_{l=0}^{\infty} 2^{lQ/r}( M (|f|^r)(x))^{1/r} \left( \int_G |\tilde{K}_{k,l}(y)|^{r'} dy \right)^{1/r'}  
\\
& \lesssim &  \| \tilde{m}_k \|_{H^s} ~\sum_{l=0}^{\infty} 2^{lQ/r} 2^{-l\alpha /2 } (M (|f|^r ) (x))^{1/r}
\\
&\lesssim & \|\tilde{m}_k\|_{H^{s}} \sum_{l=0}^{\infty} 2^{-\frac{l \epsilon}{2}} (M(|f|^r)(x))^{1/r}
\\
&\lesssim & \| \tilde{m}_k\|_{H^{s}} (M(|f|^r)(x))^{1/r}.
\end{eqnarray*}
This proves the lemma.
\end{proof}
\section{Martingales on homogeneous space}
In this section we shall recalll the martingales on homogeneous space and exploit the cancellation property arising when we compose the martingale operators and Littlewood-Paley projections.
\
In what fllows, open set $Q_{\alpha}^{k}$ will  role as dyadic cubes of side-lengths $2^{-k}$ (or more precisely, $\delta^{k}$) with the two conventions: 1. For each $k$, the index $\alpha$ will run over some unspecified index set dependent on $k$. 2. For two sets with $Q_{\alpha}^{k+1} \subset Q_{\beta}^{k}$, we say that $Q_{\beta}^{k}$ is a parent of $Q_{\alpha}^{k+1}$, and $Q_{\alpha}^{k+1}$ a child of $Q_{\beta}^{k}$. 
\
\begin{thm}[Theorem 14 in \cite{C2}]\label{moh} Let $X$ be a space of homogeneous type. Then there exists a family of subset $Q_{\alpha}^{k} \subset X$, defined for all integers $k$, and constants $\delta, \epsilon > 0, C < \infty$ such that
\begin{itemize}
\item[-] $\mu ( X \setminus \cup_{\alpha} Q_{\alpha}^{k} ) = 0 ~\forall k$
\item[-] for any $\alpha, \beta, k, l$ with $l \geq k$, either $Q_{\beta}^{l} \subset Q_{\alpha}^{k} $ or $Q_{\beta}^{l} \cap Q_{\alpha}^{k} = {\O}$
\item[-] each $Q_{\alpha}^{k}$ has exactly one parent for all $k \geq 1$
\item[-] each $Q_{\alpha}^{k}$ has at least one child
\item[-] if $Q_{\alpha}^{k+1} \subset Q_{\beta}^{k}$ then $\mu(Q_{\alpha}^{k+1}) \geq \epsilon \mu (Q_{\beta}^{k})$
\item[-] for each $(\alpha, k)$ there exists $x_{\alpha,k} \in X$ such that $B(x_{\alpha,k}, \delta^{k}) \subset Q_{\alpha}^{k} \subset B(x_{\alpha,k}, C \delta^{k})$.
\end{itemize}
Moreover,
\begin{eqnarray}\label{muyq}
\mu\{ y \in Q_{\alpha}^{k}: \rho (y, X \setminus Q_{\alpha}^{k}) \leq t \delta^k \} \leq C t^{\epsilon} \mu(Q_{\alpha}^k) ~for ~ 0 < t \leq 1, ~ for~all ~\alpha, k.
\end{eqnarray}
\end{thm}
Now we can define the expectation operator
\begin{eqnarray*}
\mathbb{E}_k f (x) = \mu(Q_{\alpha}^{k})^{-1} \int_{Q_{\alpha}^{k}} f d\mu \quad \textrm{for} ~ x \in Q_{\alpha}^{k},
\end{eqnarray*}
and  the martingale operator $ \mathbb{D}_k f (x) = \mathbb{E}_{k+1} f (x) - \mathbb{E}_k f (x).$ We set $S(f)$ be the square function: 
\begin{eqnarray*}
S (f) = (\sum_{k \geq 1} | \mathbb{D}_k f(x)|^2)^{1/2}.
\end{eqnarray*}
 Let us recall the good $\lambda$ inequality.
\begin{lem}[\cite{CWW} Corollary 3.1]  There is a constant $c_d > 0$ so that for all $\lambda >0, ~ 0 < \epsilon < \frac{1}{2}$, the following inequality holds.
\begin{eqnarray}\label{measx}
\textrm{meas} ( \{ x : \sup_{k \geq 1} | \mathbb{E}_k g(x) - \mathbb{E} g(x)| > 2\lambda, S(g) < \epsilon \lambda \}) 
\\
\leq C \exp( - \frac{C_d}{\epsilon^2}) \textrm{meas} ( \{ x : \sup_{k \geq 1} | \mathbb{E}_k g(x)| > \lambda \}) ;
\end{eqnarray}
\end{lem} 
In fact, the above lemma in \cite{CWW} is written for the martingales on Euclidean space, however, the proof works for our martingales on homogeneous space as well. We choose a bump function $\psi \in C_0^{\infty}$ which is supported on $[\frac{1}{4},4]$ and equal to 1 on $[\frac{1}{2}, 2 ]$. Set $\psi_j (\xi) = \psi(2^{-j} \xi)$. Then, since the support of $m_j (\xi) = \phi(2^{-j} \xi) m(\xi)$ is contained in $[1/2,2]$, we have $m_j (\xi) = \psi_j^{2}(\xi) m_j(\xi)$. Thus we get
\begin{eqnarray*}
m_j (L) = \psi_j (L) m_j (L) \psi_j (L),
\end{eqnarray*}
and  
\begin{eqnarray}\label{Cancel}
\mathbb{D}_k (m(L) f) &=& \mathbb{D}_k ( \sum_{ j \in \mathbb{Z}} m_j (L) f) =
\sum_{j \in \mathbb{Z}} \mathbb{D}_k (\psi_j (L) m_j (L) \psi_j (L) f).
\end{eqnarray}
The multiplier $\psi_j (L)$ is usually called Littlewood-Paley projection in the literature. We shall exploit the cancellation property between the projections and the martingale operators.
\begin{lem}\label{there}~
\begin{enumerate}
\item[(i)] There exist $a>0$ such that $| \mathbb{E}_k ( \psi_n(L) f ) (x)| \lesssim 2^{(- (\log_2 \delta) k - n/2)a } M_q f (x)$ holds uniformly for $ n/2 > (-\log_2 \delta) k + 10$.
\item[(ii)] There exist $a>0$ such that $| \mathbb{D}_k (\psi_n(L)  f ) (x) | \lesssim 2^{((\log_2 \delta)k + n/2)a} M_q f(x)$ holds uniformly for  $ n/2 < (-\log_2 \delta) k - 10$. 
\end{enumerate}
From these two estimates we have 
\begin{eqnarray*}
|\mathbb{D}_k (\psi_n (L)f ) (x) | \lesssim 2^{-|(\log_2 {\delta}) k + n/2|} M_q f(x),\quad \forall (n,k) \in \mathbb{Z}^2.
\end{eqnarray*}
\end{lem} 
\begin{proof}
For $n \in \mathbb{Z}$ we denote by $K_n : G \rightarrow \mathbb{R}$ the kernel of $\psi_n (L)$, that is, it satisfies
\begin{eqnarray*}
\psi_n (L) f = K_n * f \qquad \forall f \in S(G).
\end{eqnarray*}
Let us denote $K_1$ by $K$. Observe that we have $\int_G K(x) dx =0$ since  the support of $\psi$ is away from zero.  Moreover, it follows from \cite{FS2}[Lemma 6.36] that  
\begin{eqnarray}\label{kx1xn}
K(x) \lesssim (1+|x|)^{-N} \quad \textrm{ for any}~  N >0. 
\end{eqnarray}From Lemma \ref{dilation} we have $K_n (x) = 2^{Qn/2} K(2^{n/2} x)$. For $x \in G$, find $Q_{\alpha}^{k}$ such that $x \in Q_{\alpha}^{k}$, then we have
\begin{equation}\label{knlfx}
\begin{split}
\mathbb{E}_k (\psi_n(L)  f) (x) =& \frac{1}{\mu (Q_{\alpha}^{k})} \int_{Q_{\alpha}^{k}} (\psi_n(L)  f) (y) dy
\\
=&\frac{1}{\mu(Q_{\alpha}^{k})} \int_{Q_{\alpha}^{k}} \left[ \int_G 2^{Qn/2} K (2^{n/2} (y\cdot z^{-1}) ) f (z) dz \right] dy
\\
=& \frac{1}{\mu(Q_{\alpha}^{k})} \int_G \left[ \int_{Q_{\alpha}^{k}} 2^{Qn/2} K (2^{n/2} (y\cdot z^{-1}) ) dy \right] f (z) dz.
\end{split}
\end{equation}
Assume that $n/2 > (-\log_2 \delta) k + 10$ holds. We split the space $G$ into the following disjoint subsets: 
\begin{itemize}
\item[-] $B = \{ z : \dist(z, \partial Q^{k}_{\alpha}) \leq 2^{- [ (-\log_2 \delta) k + \frac{m}{2}]} \}$ with $m = (\log_2 \delta) k + \frac{n}{2}$.
\item[-] $ A_1 = Q_{\alpha}^{k} \cap B^{c}$
\item[-] $ A_2 = (Q_{\alpha}^{k})^{c} \cap B^{c}$.
\end{itemize}
Then $G= B \cup A_1 \cup A_2$ and we have $f = f_{A_1} + f_{A_2} + f_B := f \chi_{A_1} + f \chi_{A_2} + f \chi_{B}$ and so,
\begin{eqnarray*}
\mathbb{E}_k (\psi_n (L) f)(x) = \mathbb{E}_k (\psi_n (L) f_{A_1})(x) + 
\mathbb{E}_k (\psi_n (L) f_{A_2})(x)
+\mathbb{E}_k (\psi_n (L) f_{B})(x)
\end{eqnarray*}
We shall estimate each three terms.
\

\noindent$\cdot ~Estimate ~for ~f_{A_1}$. 
\\
Replacing $f$ with $f_{A_1}$ in \eqref{knlfx} we see 
\begin{equation}\label{3C}
\begin{split}
\mathbb{E}_k (\psi_n(L)  f_{A_1} (x)) =&\frac{1}{\mu (Q_{\alpha}^{k})} \int_{Q^{k}_{\alpha}} \left[ \int_G 2^{Qn/2} K (2^{n/2} (y\cdot z^{-1})) 1_{A_2} (z) f(z) dz \right] dy
\\
 =& \frac{1}{\mu(Q_{\alpha}^{k})} \int_G \left[ \int_{Q_{\alpha}^{k}} 2^{Qn/2} K (2^{n/2} (y\cdot z^{-1}) ) dy \right] \chi_{A_1} (z) f (z) dz.
\end{split}
\end{equation}
Because $\int_G K=0$, we have
\begin{equation}\label{3D}
\begin{split}
\left| \int_{Q_{\alpha}^{k}} 2^{Qn/2} K(2^{n/2} (y\cdot z^{-1}) ) dy \right| =&  \left|\int_{(Q_{\alpha}^{k})^{c}} 2^{Qn/2} K (2^{n/2} (y \cdot z^{-1}) ) ~dy \right|
\\
\leq & \int_{(Q_{\alpha}^{k})^{c}} 2^{Qn/2} |K (2^{n/2} (y \cdot z^{-1}))| ~dy
\\
\leq & \int_{|w| \geq 2^{- [ (-\log_2 \delta) k + \frac{m}{2}]}} 2^{Qn/2} |K (2^{n/2} w)| ~dw
\\
\leq & \int_{|w| \geq 2^{m/2}} |K(w)| ~dw \leq 2^{-m/2 c},
\end{split}
\end{equation}
where the second inequality holds since $z \in A_1 = Q_{\alpha}^{k} \cap B^{c}$ and $y \in (Q_{\alpha}^{k})^{c}$. 
From \eqref{3C} and \eqref{3D} we get
\begin{eqnarray*}
|\mathbb{E}_k (\psi_n(L)  f(x))| &\leq& \frac{1}{\mu(Q_{\alpha}^{k})} \int_G 2^{-mc/2} 1_{A_1}(z) f(z) dz
\\
&\leq & 2^{-mc/2} M f (x).
\end{eqnarray*}

\noindent $\cdot~ Estimate ~for~ f_{A_2}$.  
\

We have 
\begin{eqnarray}\label{ekpnl}
\mathbb{E}_k ( \psi_n(L)  f_{A_2}(x)) &=& \frac{1}{\mu (Q_{\alpha}^{k})} \int_{Q^{k}_{\alpha}} \left[ \int_G 2^{Qn/2} K (2^{n/2} (y\cdot z^{-1})) 1_{A_2} (z) f(z) dz \right] dy.
\end{eqnarray}
Note that we have $|(y \cdot z^{-1})| \geq 2^{- [ (-\log_2 \delta) k + \frac{m}{2}]}$ for $ z \in A_2 =(Q_{\alpha}^{k})^{c} \cap B^{\cap}$ and $y \in Q_{\alpha}^{k}$. Thus we get $|2^{n/2} (y\cdot z^{-1})| \geq 2^{n/2 + (\log_2 \delta)k -\frac{m}{2}} = 2^{\frac{m}{2}}$ in the above formula. Then, using \eqref{kx1xn} we deduce that
\begin{eqnarray*}
 \int_G 2^{Qn/2}\left| K (2^{n/2}(y\cdot z^{-1})\right| dz \lesssim \int_{ |x| \geq 2^{m/2}} (1 + |x|)^{-3N} dx \lesssim 2^{-mN}. 
\end{eqnarray*}
Using this, we get 
\begin{eqnarray*}
\left|\int_G 2^{Qn/2} K (2^{n/2} (y\cdot z^{-1}) ) 1_{A_2} f (z) dz \right| \leq M f (y) \cdot 2^{-m N}
\end{eqnarray*}
with a sufficiently large $N >0$. Inject this into \eqref{ekpnl}, we have
\begin{eqnarray*} 
\mathbb{E}_k (\psi_n(L)  f_{K_2} (x)) &\leq& \frac{1}{\mu (Q_{\alpha}^{k})} \int_{Q^k_{\alpha}} M f(y)\cdot 2^{-m N} dy 
\\
&\lesssim& M f (x) \cdot 2^{-m N}.
\end{eqnarray*}
\noindent$\cdot ~Estimate ~ for ~ f_B$.
\\
We have 
\begin{equation}\label{fb}
\begin{split}
|\mathbb{E}_k (\phi_n (L) f_B )(x)| &= \frac{1}{\mu(Q_{\alpha}^{k})} \left| \int_B \left[ \int_{Q_{\alpha}^{k}} 2^{Qn/2} K (2^{n/2} (y \cdot z^{-1}) dy \right] f(z) dz \right|
\\
&\leq \frac{1}{\mu (Q_{\alpha}^k)} \int_B \int_{Q_{\alpha}^{k}} 2^{Qn/2}|K(2^{n/2}(y z^{-1}))| dy f(z) dz
\\
&\leq \frac{1}{\mu(Q_{\alpha}^{k})} \int_B \left( \int_G 2^{Qn/2} |K(2^{n/2}(y))| dy\right) |f(z)| dz
\\
&\leq \frac{C}{\mu(Q_{\alpha}^{k})} \int_B |f(z)| dz.
\end{split}
\end{equation}
Recall from \eqref{muyq} that we have $\mu (B) \lesssim \mu(Q_{\alpha}^{k}) 2^{-\frac{m}{2}\epsilon}$. Thus we can estimate \eqref{fb} as
\begin{eqnarray*}
| \mathbb{E}_k (\phi_n (L) f_B)(x)| &\lesssim &\frac{1}{\mu(Q_{\alpha}^{k})} \int_B |f(z) | dz 
\\
&\lesssim& \frac{1}{\mu(Q_{\alpha}^{k})} \mu (B)^{1/q'} \left( \int_B |f(z)|^{q} dz\right)^{1/q}
\\
&\lesssim& 2^{-\frac{\epsilon}{2q'} m} \left( \frac{1}{\mu(Q_k^{\alpha})} \int_B |f(z)|^q dx \right)^{1/q} 
\\
&\lesssim& 2^{-\frac{\epsilon}{2q'}m} M_q f(x).
\end{eqnarray*}
From the above three estimates, for $a = \min (\frac{c}{2}, \frac{\epsilon}{2q'})$ we get
\begin{eqnarray*}
|\mathbb{E}_k (\phi_n (L) f)| &=&| \mathbb{E}_k (\phi_n (L)( f_{A_1}+ f_{A_2} + f_B))(x) |
\\
&\lesssim& 2^{-am} M_q f(x).
\end{eqnarray*}
It completes the proof of (i).
\

We now assume that $n/2 < (-\log_2 \delta) k - 10$ holds. Since $\mathbb{D}_k (\psi_n(L) f) =\mathbb{E}_{k+1} (\psi_n(L)  f) -\mathbb{E}_{k} (\psi_n(L)  f)$ we have
\begin{equation}    
\begin{split} \mathbb{D}_{k}& (\psi_n(L)  f) (x) \\
=& ~\frac{1}{\mu(Q_{\alpha}^{k+1})} \int_{Q_{\alpha}^{k+1}} (\psi_n(L)  f) (y) dy - \frac{1}{\mu(Q_{\alpha}^{k})} \int_{Q_{\alpha}^{k}} (\psi_n(L)  f)(y) dy
\\
=& ~\int_G  f(z) \left[ \frac{1}{\mu(Q_{\alpha}^{k+1})} \int_{Q_{\alpha}^{k+1}} 2^{Qn/2} K(2^{n/2} (y \cdot z^{-1})) dy - \frac{1}{\mu(Q_{\alpha}^{k})} \int_{Q_{\alpha}^{k}} 2^{Qn/2} K(2^{n/2} (y\cdot z^{-1})) dy\right] dz 
\\
=& ~\int_G  f(z) \left[ \frac{1}{\mu(Q_{\alpha}^{k+1})} \int_{Q_{\alpha}^{k+1}} 2^{Qn/2} \left[ K(2^{n/2} (y \cdot z^{-1})) - K(2^{n/2} (x \cdot z^{-1})) \right]dy \right] dz
\\
\quad\quad&\quad\qquad\qquad- \int_G f(z) \left[ \frac{1}{\mu(Q_{\alpha}^{k})} \int_{Q_{\alpha}^{k}} 2^{Qn/2} \left[ K(2^{n/2} (y\cdot z^{-1}))- K(2^{n/2} (x \cdot z^{-1}))\right] dy\right] dz 
\\
:=& ~A_1 + A_2,
\end{split}  
\end{equation} 
where we injected the identity $ 0 = - 2^{Qn/2} K (2^{n/2} (x\cdot z^{-1})) + 2^{Qn/2} K (2^{n/2} (x\cdot z^{-1}))$ in the third equality. For $x, y \in Q_{\alpha}^{k}$ we have  $\frac{n}{2} < (-\log_2 \delta)k -10$ and $|(yx^{-1})| \leq \delta^k$, and so $|2^{n/2} (y x^{-1})|\leq 2^{n/2} 2^{(\log_2 \delta)k}\leq 2^{-10}$. From the mean value theorem \cite[Theorem 1.33]{FS}, for a constant $\beta = \beta(G) >0$ we have that
\begin{equation}
\begin{split}
&\left| K( (2^{n/2}(y x^{-1}) \cdot 2^{n/2}(x z^{-1})) - K (2^{n/2}(xz^{-1})) \right|
\\
&\quad\quad\quad\quad\quad\quad\quad\lesssim \sum_{j=1}^{d} |2^{n/2}  (yx^{-1})|^{d_j} \sup_{|w|\leq |\beta 2^{n/2} (yx^{-1})|} \left| Y_j K (w 2^{n/2}(xz^{-1}) \right|
\\
&\quad\quad\quad\quad\quad\quad\quad\lesssim \sum_{j=1}^{d} (2^{n/2} \delta^{k})^{d_j} \sup_{|w|\leq \beta 2^{10}} \left( \beta 2^{11} + \beta  |w 2^{n/2}(x z^{-1})|\right)^{-N}
\\
&\quad\quad\quad\quad\quad\quad\quad\lesssim  (2^{n/2} \delta^{k}) \left( \beta 2^{10} + \beta  | 2^{n/2}(x z^{-1})|\right)^{-N}
\\
&\quad\quad\quad\quad\quad\quad\quad\lesssim (2^{n/2} \delta^{k}) \left( 1+ |2^{n/2}(x z^{-1})\right)^{-N}.
\end{split}
\end{equation}
Thus we get
\begin{eqnarray*}
|A_1 | &\lesssim& (2^{n/2} \delta^{k}) \int_G  2^{Qn/2} \left(1 + |2^{n/2} (xz^{-1})|\right)^{-N} f(z) dz
\\
&\lesssim & (2^{n/2} \delta^{k})  M f(x).
\end{eqnarray*}
It follows from the same argument that $|A_2| \lesssim (2^{n/2} \delta^{k}) M f(x)$. We thus obtain
\begin{eqnarray*}
|\mathbb{D}_k (\psi_n (L) f ) (x) | \lesssim  (2^{n/2} \delta^k) M f (x),
\end{eqnarray*}
which completes the proof of (ii).
\end{proof}
We set $\mathcal{M} = M \circ M \circ M$ and
\begin{eqnarray}\label{grfpc}
G_r (f) = ( \sum_{k \in \mathbb{Z}} (\mathcal{M} (|L_k f |^r))^{2/r} )^{1/2}.
\end{eqnarray}
We shall need the following inequality of Fefferman-Stein \cite{FS}:
\begin{eqnarray}\label{grfpc}
\| G_r (f) \|_p \leq C_{p,r} \| f \|_p , \qquad 1 < r <2, ~ r < p < \infty.
\end{eqnarray}
This inequality is very useful for us because the square function $S$ is bonded by $G_r$ as follow
\begin{lem}\label{prop} If $1 < r \leq \infty$ and $\alpha > \frac{Q}{r}$, then we have  
\begin{eqnarray}\label{smlf1} S(m(L)f)(x) \leq A_r \| m \|_{L_2^{\alpha}} G_r (f) (x)\quad \forall x \in G.
\end{eqnarray}
 If we further assume that $m(\xi) =0$ for $|\xi| \leq N $, we get 
\begin{eqnarray}\label{eomlf}
\mathbb{E}_0 (m(L) f) (x) \lesssim 2^{-N} \| m\|_{L_2^{\alpha}} G_r (f) (x).\end{eqnarray}
\end{lem}
\begin{proof}
Using Lemma \ref{there} we get
\begin{eqnarray*}
|\mathbb{B}_k (m(L)f)(x)| &=& \left| \sum_{n \in \mathbb{Z}} \mathbb{B}_k ( \psi_n(L)  m_n (L) \psi_n (L) f )(x) \right|
\\
& \lesssim & \| m\|_{L_2^{\alpha}}\sum_{n \in \mathbb{Z}} 2^{- | k |\log_2 \delta| -n |} M_r ( \psi_n(L)  f).
\end{eqnarray*}
We then use Cauchy-Scwartz inequality, 
\begin{eqnarray*}
|\mathbb{B}_k (m(L)f)(x)|^2 & \lesssim & \left( \sum_{n \in \mathbb{Z}} 2^{- |k|\log_2 \delta| - n |}\right) \sum_{n \in \mathbb{Z}} 2^{- ||k\log_2 \delta| -n| } (M_r (\psi_n(L)  f ))^2
\\
&\lesssim& \sum_{n \in \mathbb{Z}} 2^{- ||k\log_2 \delta| -n| } (M_r (\psi_n(L)  f ))^2.
\end{eqnarray*}
Therefore we get,
\begin{eqnarray*}
S(m(L)f) (x)^2 &=& \sum_{k=1}^{\infty} |\mathbb{B}_k (m(L)f)(x)|^2 
\\
&\lesssim& \sum_{k=1}^{\infty} \sum_{n \in \mathbb{Z}} 2^{- ||k\log_2 \delta| -n| }|M_r (\psi_n(L)  f )(x)|^2
\\
&\lesssim& \sum_{n \in \mathbb{Z}} |M_r (\psi_n (L)f)(x)|^2.
\end{eqnarray*}
This proves \eqref{smlf1}. We now suppose that $m(\xi) =0$ for $|\xi |\leq N$. Then, if follows again from Lemma \ref{there} that
\begin{eqnarray*}
|\mathbb{E}_0 (m(L)f)(x)| &=& \left| \sum_{n \geq N-1} \mathbb{E}_0 ( \psi_n(L)  m_n (L) \psi_n (L) f )(x) \right|
\\
& \lesssim & \| m\|_{L_2^{\alpha}}\sum_{n \geq N-1} 2^{-  n } M_r ( \psi_n(L)  f)(x)
\\
& \lesssim & \| m\|_{L_2^{\alpha}} 2^{-N} M_r (\psi_n (L)f)(x),
\end{eqnarray*}
which proves  \eqref{eomlf}.
\end{proof}
\begin{proof}[proof of Theorem \ref{thm1}]
We need to bound
\begin{eqnarray}\label{1inti}
\left\| \sup_{1 \leq i \leq N } |T_i f|\right\|_p = \left( p 4^p \int^{\infty}_0 \lambda^{p-1} \textrm{meas} \left( \left\{ x : \sup_{i} | T_i f(x)|  > 4\lambda\right\}\right) d\lambda \right)^{1/p}
\end{eqnarray}
by some constant time of $ \sqrt{\log (N+1)} \| f\|_p$. We shall follow the argument of Grafakos-Honzig-Seeger \cite{GHS} with our lemmas. By Lemma \ref{prop} we have the pointwise bound
\begin{eqnarray}\label{SG}
S( T_i f ) \leq A_r B G_r (f).
\end{eqnarray}
We split the level set in \eqref{1inti} as
\begin{eqnarray*}
\left\{ x : \sup_{1 \leq i \leq N} |T_i f (x)| > 4 \lambda \right\} \subset E_{\lambda} \cup F_{\lambda},
\end{eqnarray*}
where
\begin{eqnarray*}
E_{\lambda} &=& \{ x : \sup_{1 \leq i \leq N} | T_i f (x) - \mathbb{E}_0 T_i f (x)| > 2\lambda \},
\\
F_{\lambda} &=& \{ x : \sup_{1 \leq i \leq N} | \mathbb{E}_0 T_i f (x) > 2\lambda \}.
\end{eqnarray*}
FAs for the set$F_{\lambda}$ we use \eqref{eomlf} to deduce that
\begin{equation}\label{p2pp1}
\begin{split}
\left( p 2^{p} \int^{\infty}_{0} \lambda^{p-1} \textrm{meas}(F) d\lambda \right)^{1/p} =& \left\| \sup_{1 \leq i \leq N} |\mathbb{E}_0 T_i f(x)| \right\|_p  
\\
\leq& \sum_{1\leq i \leq N} \left\|\mathbb{E}_0 T_i f(x) \right\|_p
\\
\leq& B N 2^{-N/r} \left\| M (|f|^r) \right\|_p.
\end{split}
\end{equation}
As for the set $E_{\lambda}$ we split it one more as  $E_{\lambda} \subset E_{\lambda,1} \cup E_{\lambda,2}$ with
\begin{eqnarray*}
E_{\lambda,1} &=& \left\{ x : \sup_{1 \leq i \leq N} | T_i f (x) - \mathbb{E}_0 T_i f (x)| > 2\lambda, G_r (f) (x) \leq \frac{\varepsilon_N \lambda}{A_r B}\right\},
\\
E_{\lambda, 2} &=& \left\{ x : G_r (f)(x) > \frac{\varepsilon_N \lambda}{A_r B}\right\},
\end{eqnarray*}
where we set
\begin{eqnarray*}
\epsilon_N : = \left( \frac{c_d} {10 \log (N+1)}\right)^{1/2}.
\end{eqnarray*}
By \eqref{SG}, 
\begin{eqnarray*}
E_{\lambda,1} \subset \bigcup_{i=1}^{N} \left\{ x : |T_i f(x)-\mathbb{E}_0 T_i f (x) | > 2\lambda, S(T_i f ) \leq \varepsilon_N \lambda \right\}.
\end{eqnarray*}
By using the inequality \eqref{measx} we have
\begin{eqnarray*}
\textrm{meas}(E_{\lambda,1}) &\leq& \sum_{i=1}^{N} \textrm{meas} ( \{ x : |T_i f (x) - \mathbb{E}_0 T_i f (x)| > 2\lambda, S(T_i f) \leq \varepsilon_N \lambda\})
\\
&\leq & \sum_{i=1}^{N} C \exp ( - \frac{c_d}{\varepsilon_N^2}) \textrm{meas} \left( \left\{ x : \sup_{k} | \mathbb{E}_k (T_i f)| > \lambda \right\}\right).
\end{eqnarray*}
Therefore
\begin{equation}\label{pp1me}
\begin{split}
\left( p \int_0^{\infty} \lambda^{p-1} \textrm{meas} (E_{\lambda, 1}) d\lambda\right)^{1/p} \lesssim~ &( \sum_{i=1}^{N} \exp (-\frac{c_d}{\varepsilon_N^2} )\| \sup_{k} |\mathbb{E}_k (T_i f)\|_p^p )^{1/p}
\\
\lesssim ~&( \sum_{i=1}^{N} \exp (-\frac{c_d}{\varepsilon_N^2}) \| T_i f\|_p^p)^{1/p}
\\
\lesssim ~&B ( N \exp (-\frac{c_d}{\varepsilon_N^2}))^{1/p} \|f \|_p 
\\ \lesssim ~ &B\|f\|_p.
\end{split}
\end{equation}
Using a change of variables and \eqref{grfpc}, we have
\begin{equation}\label{arbgr}
\begin{split}
( p \int_0^{\infty} \lambda^{p-1} \textrm{meas}(E_{\lambda,2}) d\lambda)^{1/p} =& \frac{A_r B}{\varepsilon_N} \| G_r (f)\|_p
\\
\lesssim & B \sqrt{\log (N+1)} \| f\|_p.
\end{split}
\end{equation}
From \eqref{p2pp1}, \eqref{pp1me} and \eqref{arbgr}, we get the desired esimate for \eqref{1inti}. 
\end{proof}

\begin{proof}[proof of Theorem \ref{thm2}]
Set 
\begin{eqnarray} 
\mathcal{M}_m^{\textrm{dyad}} f(x) = \sup_{k \in \mathbb{Z}} | m(2^k L) f(x)|.
\end{eqnarray}
Let 
\begin{eqnarray*}
\mathcal{I_j} = \{ k \in \mathbb{Z}: w^{*} (2^{2^j}) < | \omega(k)| \leq \omega^{*} (2^{2^{j-1}})\}.
\end{eqnarray*}
We may have the partition $m = \sum_{j} m_j$ where $m_j$ has support in the union of dyadic interval $U_{k \in \mathcal{I}_j} \{ \xi : 2^{k-1} < |\xi| <2^{k+1}\}.$  We set
\begin{eqnarray*}
T_k^{j} f = m_j (2^k L) f.
\end{eqnarray*}
\

Using Lemma 3.1 in \cite{CGHS} we may take a sequence of integers $B = \{i\}$ such that $\mathbb{Z} = \cup_{n=-4^{2^j +1}}^{4^{2^j+1}} (n+B)$ and the sets $b_i + \mathcal{I}_j$ are pairwise disjoint for each fixed $j$. Now we split the supremum as
\begin{eqnarray*}
\sup_{k} |T_k f| = \sup_{|n| \leq 4^{2^j+1}} \sup_{i \in \mathbb{Z}} |T_{b_i +n} f|.
\end{eqnarray*}
We shall use the $L^p$ norm equivalence of Rademacher functions $\{r_i \}_{i=1}^{\infty}$ 
\begin{eqnarray*}
c_p \left( \sum_{i} a_i^2\right)^{1/2} \leq  \left( \int_0^{1} \left| \sum_{i=1}^{\infty} r_i (s) a_i \right|^2  ds \right)^{1/p} \leq C_p \left( \sum_{i} a_i^2\right)^{1/2}
\end{eqnarray*}
(see e.g. \cite[p. 276]{S}). Then
\begin{eqnarray*}
\left\| \sup_{ |n|\leq 4^{2^j +1}} \sup_{i>0} |T_{b_i +n}^{j} f| \right\|_p &\leq& \left\| \sup_{|n| \leq 4^{2^j +1}} \left( \sum_{i>0} |T_{b_i+n}^{j} f|^2 \right)^{1/2} \right\|_{p}
\\
&\leq &C_p \left\| \sup_{|n| \leq 4^{2^j+1}} \left( \int_0^{1} \left| \sum_{i=1}^{\infty} r_i (s) T_{b_i +n}^{j} f \right|^p ds \right)^{1/p} \right\|_p
\\
&\leq& C_p \left\| \left( \int_0^{1} \sup_{|n| \leq 4^{2^j+1}} \left| \sum_{i=1}^{\infty} r_i (s) T_{b_i +n}^{j} f \right|^p dx \right)^{1/p} \right\|_p
\\
&=&C_p \left( \int_0^{1} \left\| \sup_{|n| \leq 4^{2^j}} \left| \sum_{i=1}^{\infty} r_i (s) T_{b_i +n}^{j} f \right| \right\|_p^p ds \right)^{1/p}.
\end{eqnarray*}
We now apply Theorem \ref{thm1} to obtain
\begin{eqnarray*}
\| \mathcal{M}_{m_j}^{\textrm{dyad}} \|_{L^p \rightarrow L^p} \lesssim 2^{j/2} \omega^{*} (2^{2^{j-1}}).
\end{eqnarray*}
Now we use a calculus to get
\begin{eqnarray*}
\sup_{2^k \leq t < 2^{k+1}} | m_j( tL) f(x) | &=& \sup_{1\leq t < 2} | m_j(t 2^{k} L) f(x) | \\
&\leq& | m( 2^{k} L) f(x) | + \int_{1}^{2} \left| \frac{\partial}{\partial t} m_j(t 2^{k} L) f(x) \right| dt.
\end{eqnarray*}
One may observe
\begin{eqnarray*}
\left\| \phi (s) ({\partial}/{\partial t}) m_j ( t 2^{k} s) \right\|_{H^{\alpha }} \lesssim \sum_{l=k-1}^{k+1} \| \phi (s) m_j (2^{l} (s) \|_{H^{\alpha +1}}
\end{eqnarray*}
holds uniformly for $1 \leq t \leq 2$. Therefore, the boundedness of $\mathcal{M}_m$ follows from that one for $\mathcal{M}^{\textrm{dyad}}_m$.
\end{proof}

\section{Maximal multipliers on product spaces}
In this section we study the maximal multipliers on product spaces of stratified groups. In addition we also obtain a similar result for the joint spectral multipliers on the Heisenberg group.  

\
Let $G$ be the direct product of $n$ stratified groups $G_1,\cdots ,G_n$ endowed with sub-Laplacians $L_1, \cdots, L_n$. We set $L_1^{\sharp}, \cdots, L_n^{\sharp}$ be the lifted sub-Laplacians  on $G$. Then $L_1^{\sharp}, \cdots, L_n^{\sharp}$ mutually commute and so their spectral measures $dE_1 (\xi), \cdots, dE_n (\eta)$ also mutually commute. Thus, for a bounded function $m$ on $\mathbb{R}_{+}^{n}$, we can define the joint spectral multiplier 
\begin{eqnarray*}
m(L_1^{\sharp}, \cdots, L_n^{\sharp}) = \int_{\mathbb{R}_{+}^n} m(\xi_1, \cdots, \xi_n) ~d E_1 (\xi)\cdots d E_n (\xi).
\end{eqnarray*}
Under the following assumption on $m$,
\begin{eqnarray}\label{5A}
| (\xi_1 \partial_{\xi_1})^{\alpha_1} \cdots (\xi_n \partial_{\xi_n})^{\alpha_n} m (\xi_1,\cdots,\xi_n)| \leq C_{\alpha}
\end{eqnarray}
for all $\alpha_j \leq N$ with $N$ large enough, M\"uller-Ricci-Stein \cite{MRS} proved that $m(L_1^{\sharp},\cdots, L_n^{\sharp})$ is bounded on $L^p (G)$. For each group $G_k$ we may endow the martingales with index set $\{Q_{\alpha}^{k,j}: j \in \mathbb{N}_0, \alpha \in I_j \}$ given by Theorem \ref{moh}. For $1 \leq j \leq n$ and $k \in \mathbb{N}_0$ we set the martingale operator $\mathcal{E}_k^{j} : S(G) \rightarrow S(G)$ as 
\begin{eqnarray*}
\mathcal{E}_k^j f (x_1,\cdots, x_n) = \frac{1}{|Q_{\alpha}^{k,j}|}\int_{Q_{\alpha}^{k,j}} f(x_1,\cdots,x_n) dx_j,  
\end{eqnarray*}
where $\alpha \in I_j$ is a unique index such that $ x_j \in Q_{\alpha}^{k,j}$. Denote $\mathcal{E}_{k}^0$ by  $\mathcal{E}_{k}$. For $1\leq s \leq n$, $1\leq n_1 < n_2 \cdots <n_s \leq n$ and $(k_1,\cdots, k_s) \in \mathbb{N}_0^{s}$, we define  the multi-expectation function 
\begin{eqnarray*}
\mathcal{E}_{k_1,\cdots,k_s}^{n_1,\cdots,n_s} := \mathcal{E}_{k_1}^{n_1}\cdots \mathcal{E}_{k_s}^{n_s}
\end{eqnarray*}
and the multiple martingale operator 
\begin{eqnarray*}
D_{k_1,\cdots,k_s}^{n_1,\cdots,n_s} g = \sum_{\substack{ a_j \in [0,1]\\m_j - a_j \geq 0}}  (-1)^{a_1 + \cdots + a_s} ~\mathcal{E}_{k_1 - a_1, \cdots, k_s -a_s}^{n_1,\cdots,n_s} g.
\end{eqnarray*}
For $f \in S(G)$ we see
\begin{equation}
\begin{split}
&D_{k_1,\cdots,k_t+1, \cdots,k_s}^{n_1,\cdots,n_s} g + D_{k_1,\cdots,k_j,\cdots,k_s}^{n_1,\cdots,n_s} g 
\\
&\quad\quad\quad= \sum_{\substack{a_j \in [0,1]\\ k_j - a_j \geq 0}} (-1)^{a_1 + \cdots + a_s} (\mathcal{E}_{k_1 -a_1 ,\cdots,k_t+1-a_t,\cdots,k_s -a_s}^{n_1,\cdots,n_s} g + \mathcal{E}_{k_1 -a_1,\cdots,k_s-a_s}^{n_1,\cdots,n_s} g)
\\
&\quad\quad\quad= \sum_{\substack{a_j \in [0,1]\\ k_j - a_j \geq 0}} (-1)^{a_1 + \cdots + a_{s-1}} (\mathcal{E}_{k_1 -a_1 ,\cdots,k_t+1,\cdots,k_{s}-a_{s}}^{n_1,\cdots,n_s}  - \mathcal{E}_{k_1 -a_1,\cdots,k_t -1,\cdots,k_s -a_s}^{n_1,\cdots,n_s} ),
\end{split}  
\end{equation}
where we use an abuse of notation that $\mathcal{E}_{k_1 -a_1,\cdots,k_{t-1}-a_{t-1},k_t -1,\cdots,k_s -a_s}^{n_1,\cdots,n_s}=0 $ if $k_{t}-1 =-1$. We use it to get
\begin{eqnarray}\label{dkn}
\sum_{k_j =0}^{\infty} D_{k_1,\cdots,k_s}^{n_1,\cdots,n_s} g(x) = D_{k_1,\cdots,\widehat{k}_j,\cdots,k_s}^{n_1,\cdots,\widehat{n}_j,\cdots,n_s} g(x),
\end{eqnarray}
where  $\widehat{a}$ denote absence of $a$, i.e., 
\begin{eqnarray*}
(n_1,\cdots,\widehat{n}_j,\cdots, n_s) = (n_1,\cdots,n_{j-1},n_{j+1},\cdots,n_s) \in \mathbb{N}^{s-1}.
\end{eqnarray*}
In what follows we shall use the notation  that
\begin{eqnarray*}
\sum_{k_{j_1},\cdots,k_{j_m}} := \sum_{k_{j_1}=1}^{\infty} \cdots \sum_{k_{j_m} =1}^{\infty}.
\end{eqnarray*}
For $1 \leq m \leq n$, we simply denote $D_{k_1,\cdots,k_m}$ for $D_{k_1,\cdots,k_m}^{1,2,\cdots,m}$. Using \eqref{dkn} $n$ times we obtain
\begin{eqnarray*}
\sum_{k_1,\cdots,k_n} D_{k_1,\cdots,k_n} f(x) =f(x).
\end{eqnarray*}
Set
\begin{eqnarray*}
\mathcal{A} f := (1-\mathcal{E}^1)\cdots (1-\mathcal{E}^n) f.
\end{eqnarray*}
Denote $\mathcal{A}(S(G))$ be the image of $S(G)$ under the operator $\mathcal{A}$ and $\mathcal{A}_j (S(G)) := (1-\mathcal{E}^j )(S(G))$ be the image of $S(G)$ under the operator $1-\mathcal{E}^{j}$. Note that, for $1\leq j \leq N$ we have
\begin{eqnarray*}
\mathcal{E}^{j} g = 0 \qquad \forall g \in \mathcal{A}_j (S(G)).
\end{eqnarray*} 
For $2\leq m \leq n+1$ we set the intermediate square functions $S_m$ and the maximal intermediate squre function $S_m^{*}$ defined by Honzik \cite{H} which is a general version of the double square functions defined by Pipher \cite{P},
\begin{eqnarray}\label{sm}
S_m f = \Biggl( \sum_{\substack{k_1,\cdots,k_{m-1}}} \Biggl( \sum_{\substack{k_m,\cdots,k_n }} D_{k_1,\cdots,k_n} f (x) \Biggr)^2 \Biggr)^{1/2},
\end{eqnarray}
\begin{eqnarray}\label{smstar}
S_m^* f = \sup_{r} \Biggl( \sum_{\substack{k_1,\cdots,k_{m-1} }} \Biggl( \sum_{\substack{k_m <r, k_{m+1},\cdots,k_n }} D_{k_1,\cdots,k_n} f(x) \Biggr)^2 \Biggr)^{1/2}.
\end{eqnarray}
For $m=1$ we define the following maximal function
\begin{eqnarray*}
S_1 f(x)  = \sup_{r} \left| \sum_{\substack{m_1\leq r,m_2, \cdots, m_n}} D_{m_1,\cdots,m_n} f(x) \right|. 
\end{eqnarray*} 
We then have the following lemma.
\begin{lem}[\cite{P}]\label{lemp}
Suppose $X_N^j = \sum_{q=0}^{N} d_q^j, j=1,\cdots,M$ is a sequence of dyadic martingales and set
\begin{eqnarray*}
S X_N^j = \Biggl( \sum_{q}^{N} (d_q^j)^2 \Biggr)^{1/2}
\end{eqnarray*}
be the square function of $X_N^j$. Then
\begin{eqnarray*}
\int \exp\Biggl( \sqrt{1 + \sum_{j=1}^{M} (X_N^j)^2 }- \sum_{j=1}^{M} (S X_N^j)^2 \Biggr) dx \leq e. 
\end{eqnarray*}
\end{lem}
Based on this lemma, Pipher \cite{P} obtained a good $\lambda$ inequality for two folds product spaces and Honzik generalized it  to  general product spaces with arbitrary $n \in \mathbb{N}$. Here we shall state the  good $\lambda$ inequality on product spaces in the following lemma, but we impose the condition that $g \in A_m (S(G))$ instead of $g \in \mathbb{E}_{N,\cdots,N} (S(G))$ as in \cite[Lemma 2]{H}. 
\begin{lem}\label{for2m}
Let $2 \leq m \leq n$ and $x_1,\cdots, \widehat{x}_m,\cdots,x_n \in G_1 \times \cdots \widehat{G}_m \cdots \times G_n$, there exist constants $C>0$ and $c>0$ such that
\begin{eqnarray*}
|\{ x_m  \in G_m : S_m^{*} (g(x_1, \cdots, x_n)) > 2\lambda ; S_{m+1} g(x_1,\cdots,x_n) < \epsilon \lambda \}|
\\
\leq C e^{-c/{\epsilon^2}} |\{ x_m : S_m^{*} g (x_1,\cdots,x_n) > \lambda \}|.
\end{eqnarray*}
holds for any $0 < \epsilon < 1/10$, $0 <\lambda < \infty$  and $g \in \mathcal{A}_m (S(G))$. The constants $C$ and $c$ are independent of $(x_1,\cdots, \hat{x}_m,\cdots, x_n)$.
\end{lem}
\begin{proof}
Since $g \in \mathcal{A}_m (S(G))$ we have $\mathcal{E}_m g =0$. Thus, for any $(k_1, \cdots, \widehat{k}_m,\cdots,k_n) \in \mathbb{N}_0^{n-1}$ we get
\begin{eqnarray*}
D_{k_1,\cdots,k_n} g(x) = \left( D_{k_1,\cdots,\widehat{k}_m,\cdots,k_n}^{1,\cdots \widehat{m} ,\cdots, n} \circ \mathcal{E}_m \right)g (x) =0 \quad \textrm{for} ~k_m =0. 
\end{eqnarray*}
We thus have
\begin{equation}
\begin{split}
&\Biggl( \sum_{k_1,\cdots,k_{m-1}} \Biggl(\sum_{k_m <1,k_{m+1},\cdots,k_n} D_{k_1,\cdots,k_n} g(x) \Biggr)^2 \Biggr)^{1/2} 
\\
&\quad\quad\quad\quad\quad\quad\quad\quad\quad\quad\quad\quad =\Biggl( \sum_{k_1,\cdots,k_{m-1}} \Biggl(\sum_{k_m =0,k_{m+1},\cdots,k_n} D_{k_1,\cdots,k_n} g(x) \Biggr)^2 \Biggr)^{1/2} = 0.
\end{split}
\end{equation}
Therefore, for $x_m\in \{ x_m \in G_m : S_m^{*}g (x)>\lambda \}$, we can find a minimal integer $r \geq 2$ such that
\begin{eqnarray}\label{k1km1}
\Biggl( \sum_{k_1,\cdots,k_{m-1}} \Biggl( \sum_{k_m <r,k_{m+1}, \cdots,k_n} D_{k_1,\cdots,k_n} {g} (x)\Biggr)^{2}\Biggr)^{1/2} > \lambda.
\end{eqnarray}
From the property of martingales, there exists a unique index $\alpha$ such that $x_m \in Q_{r,\alpha}^{m}$. Then \eqref{k1km1} can be written as follows.
\begin{eqnarray*}
\Biggl(\sum_{k_1,\cdots,k_{m-1}} \Biggl( D_{k_1,\cdots,k_{m-1}} \oint_{Q_{r,\alpha}^{m}} g\Biggr)^2\Biggr)^{1/2} > \lambda,
\end{eqnarray*}
where $\oint_{Q} dx$ denote the average integral  $\frac{1}{|Q|} \int_Q dx$. It follows that the set $\{ x_m \in G_m: S_m^{*}g(x)>\lambda \}$ consists of such maximal martingales $\{Q_{r_j,\alpha_j}^{m}\}_{j \in I}$ with $ r_j \geq 2$, where $I$ is an index set and $Q_{r_j,{\alpha}_j}^{m}$'s are mutually disjoint. 
\

Choose a set $Q_{r,\alpha}^{m} \subset \{Q_{r_j,\alpha_j}^{m}\}_{j\in I}$ such that $Q_{r,\alpha}^{m} \cap \{ x_m \in G_m : S_{m+1} g(x) \leq \epsilon \lambda \} \neq 0$. Then, we claim that, for any $x_m \in Q_{r,\alpha}^{m} \cap \{ x_m \in G_m : S_{m+1} g(x) \leq \epsilon \lambda\},$ we have
\begin{eqnarray}\label{5B}
\Biggl( \sum_{k_1,\cdots,k_{m-1}} \Biggl( \sum_{k_m <r, k_{m+1}, \cdots,k_n} D_{k_1,\cdots,k_n} {g} (x)\Biggr)^{2} \Biggr)^{1/2} \leq (1+\epsilon) \lambda.
\end{eqnarray}
Suppose not with a view to a contradiction, that is, 
\begin{eqnarray*}
\Biggl( \sum_{k_1,\cdots,k_{m-1}} \Biggl( \sum_{k_m <r, \cdots,k_n} D_{k_1,\cdots,k_n} {g} (x)\Biggr)^{2} \Biggr)^{1/2} > (1+\epsilon) \lambda.
\end{eqnarray*}
Since $ S_{m+1} g(x) \leq \epsilon \lambda$, we get
\begin{eqnarray*}
\Biggl( \sum_{k_1,\cdots,k_{m-1}} \Biggl( \sum_{k_m =r, k_{m+1}, \cdots,k_n} D_{k_1,\cdots,k_n} {g} (x)\Biggr)^{2} \Biggr)^{1/2} <  \Biggl( \sum_{k_1,\cdots,k_{m}} \Biggl( \sum_{k_{m+1}, \cdots,k_n} D_{k_1,\cdots,k_n} {g} (x)\Biggr)^{2} \Biggr)^{1/2} < \epsilon \lambda.
\end{eqnarray*}
Thus we have
\begin{equation*}
\begin{split}
&\Biggl( \sum_{k_1,\cdots,k_{m-1}} \Biggl( \sum_{k_m <r-1,k_{m+1}, \cdots,k_n} D_{k_1,\cdots,k_n} {g} (x)\Biggr)^{2} \Biggr)^{1/2}
\\
&> \Biggl( \sum_{k_1,\cdots,k_{m-1}} \Biggl( \sum_{k_m <r, k_{m+1},\cdots,k_n} D_{k_1,\cdots,k_n} {g} (x)\Biggr)^{2} \Biggr)^{1/2} - \Biggl( \sum_{k_1,\cdots,k_{m-1}} \Biggl( \sum_{k_m = r, k_{m+1},\cdots,k_n} D_{k_1,\cdots,k_n} {g} (x)\Biggr)^{2} \Biggr)^{1/2} 
\\
&> (1+\epsilon) \lambda -\epsilon \lambda = \lambda. ~\quad\quad\quad\quad\qquad\qquad\quad\quad\quad\quad\quad \mbox{~}\mbox{~}
\end{split}
\end{equation*}
This means that $r-1$ also satisfies the condition \eqref{k1km1}.It contradicts to the minimality of $r$. Thus the inequality \eqref{5B} must hold.
\

We now define the subset $q_{r,\alpha}^{m}  \subset Q_{r,\alpha}^{m}$ as 
\begin{eqnarray}\label{setz}
q_{r,\alpha}^{m} = \{ x_m \in Q_{r,\alpha}^{m}: S_{m+1} {g}(x) \leq \epsilon \lambda \quad \textrm{and}\quad  S_m^{*} {g}(x) > 2\lambda ~ \}. 
\end{eqnarray}
For each $x_m \in {q}_{r,\alpha}^{m}$ we take a minimal number $t_x$ such that
\begin{eqnarray}\label{5C}
\Biggl( \sum_{k_1,\cdots,k_{m-1}} \Biggl( \sum_{k_m < t_x, \cdots, k_n} D_{k_1,\cdots,k_n} {g}(x)\Biggr)^2 \Biggr)^{1/2} > 2 \lambda.
\end{eqnarray}
We then make a new martingale on $Q_{r,\alpha}^{m}$ as follows.
\begin{equation}\label{gnew}
\begin{split}
g_{new} (x) = \left\{\begin{array}{ll} \mathcal{E}_m^{t_x} g (x) - \mathcal{E}_m^{r} g(x) & \textrm{if} \quad x_m \in q_{r,\alpha}^{m},
\\
g(x) - \mathcal{E}_m^{r} g(x) & \textrm{if}\quad x_m \notin q_{r,\alpha}^{m}.
\end{array}\right.
\end{split}
\end{equation}
 Then, $\mathcal{E}_m^{r} g_{new} = 0$, and so we can use a local version of Lemma \ref{lemp} to get  
\begin{equation}\label{iexp}
\begin{split}
\int_{Q_{r,\alpha}^{m}} \exp \Biggl[ \alpha \Biggl( \sum_{k_1,\cdots,k_{m-1}} \Biggl( \sum_{k_{m+1},\cdots,k_n} D_{k_1,\cdots,k_n} g_{new}\Biggr)^2 \Biggr)^{1/2}\Biggr.\quad\quad\quad\quad\quad\quad\quad\quad\quad&
\\
 \Biggl.- \alpha^2 \sum_{k_1,\cdots,k_{m-1}} \Biggl( \sum_{k_m} \Biggl( \sum_{k_{m+1},\cdots,k_n} D_{k_1,\cdots,k_n} g_{new}\Biggr)^2 \Biggr) \Biggr]\leq e |Q_{r,\alpha}^{m}|.&
\end{split}
\end{equation}
From  the construction of $g_{new}$ we get $D_{k_1,\cdots,k_n} g_{new} =0 $ if $k_m \geq t_x$ or $k_m <r$, which implies $\sum_{k_m =1}^{\infty} D_{k_1,\cdots,k_n} g_{new} = \sum_{r \leq k_m <t} D_{k_1,\cdots,k_n} g$ and \eqref{iexp} equals to the following inequality
\begin{eqnarray}\label{iexp2}
\int_{Q_{r,\alpha}^{m}} \exp \Biggl[ \alpha \Biggl( \sum_{k_1,\cdots,k_{m-1}} \Biggl( \sum_{r \leq k_m < t_x,\cdots,k_n } D_{k_1,\cdots,k_n} {g}_{new}(x)\Biggr)^{2}\Biggr)^{1/2}\Biggr. 
\\
\Biggl.-\alpha^2 \sum_{k_1,\cdots,k_{m-1}, r\leq k_m<t_x} \Biggl( \sum_{k_{m+1},\cdots,k_n} D_{k_1,\cdots,k_n} {g}_{new}(x) \Biggr)^2 \Biggr] \leq e |Q_{r,\alpha}^{m}|.
\end{eqnarray}
For $x_m \in q_{r,\alpha}^{m}$, from \eqref{sm} we have 
\begin{eqnarray}
\Biggl( \sum_{k_1,\cdots,k_{m-1},r \leq k_m < t_x} \Biggl( \sum_{k_{m+1},\cdots,k_n} D_{k_1,\cdots,k_n} g (x) \Biggr)\Biggr)^{1/2} \leq~ S_{m+1} g(x) \leq \epsilon \lambda.
\end{eqnarray}
Using \eqref{5B} and \eqref{5C} we get 
\begin{equation*}
\begin{split}
\Biggl(\sum_{k_1,\cdots,k_{m-1}} \Biggl( \sum_{r \leq k_m<t_x,\cdots,k_n} D_{k_1,\cdots,k_n} {g}_{new}(x)\Biggr)^2 \Biggr)^{1/2}& 
\\
\geq  \Biggl( \sum_{k_1,\cdots,k_{m-1}}\Biggr. &\Biggl.\Biggl( \sum_{k_m<t_x, k_{m+1},\cdots,k_n} D_{k_1,\cdots,k_{n}} {g}_{new} (x)\Biggr)^2 \Biggr)^{1/2}
\\
- \Biggl( \sum_{k_1,\cdots,k_{n-1}}\Biggr. &\Biggl. \Biggl (\sum_{k_m < r,k_{m+1},\cdots,k_n} D_{k_1,\cdots,k_n} g_{new}(x)\Biggr)^2 \Biggr)^{1/2}
\\
\geq ~2\lambda - (1+&\epsilon)\lambda = (1-\epsilon) \lambda.
\end{split}
\end{equation*}
Therefore, for any $x_m \in q_{r,\alpha}^{m} \subset Q_{r,\alpha}^{m}$ we have
\begin{equation*}
\begin{split}
 ~&\alpha \Biggl( \sum_{k_1,\cdots,k_{m-1}}\Biggl( \sum_{r \leq k_m < t_x,\cdots,k_n } D_{k_1,\cdots,k_n} {g}_{new}(x)\Biggr)^{2}\Biggr)^{1/2}
\\
&\quad\quad\quad\quad\quad-\alpha^2 \sum_{k_1,\cdots,k_{m-1}, r\leq k_m<t} \Biggl( \sum_{k_{m+1},\cdots,k_n} D_{k_1,\cdots,k_n} {g}_{new}(x) \Biggr)^2  \geq \alpha(1-\epsilon)\lambda -\alpha^2 \epsilon^2 \lambda^2.
\end{split}
\end{equation*}
Thus, using \eqref{iexp2} we get
\begin{eqnarray*}
|q_{r,\alpha}^{m}| \exp ( \alpha (1-\epsilon)\lambda - \alpha^2 \epsilon^2 \lambda^2 ) \leq e |Q_{r,\alpha}^{m}|.
\end{eqnarray*}
Take $\alpha = \frac{1}{2\epsilon^2 \lambda}$, then we have 
\begin{eqnarray*}
|q_{r,\alpha}^{m}| \exp( \frac{(1-\epsilon)}{2 \epsilon^2} - \frac{1}{4 \epsilon^2}) = |q_{r,\alpha}^{m}| \exp( \frac{1-2\epsilon}{4 \epsilon^2} )\leq e |Q_{r,\alpha}^{m}|.
\end{eqnarray*}
We have this inequality for all $j \in I_m$ with $Q_{r_j,\alpha_j}^{m}$ and $q_{r_j,\alpha_j}^{m}$, which yields that
\begin{eqnarray*}
|\{ x_m \in G: S_m^{*}g(x) > 2\lambda,~ S_{m+1}g(x) \leq \epsilon \lambda \}| &=&\sum_{j \in I} |\{ x_m \in Q_{r_j,\alpha_j}^{m}: S_m^{*} g (x) > 2\lambda,~S_{m+1} g(x) \leq \epsilon \lambda\}|
\\
&\leq& \sum_{j \in I} e^{-\frac{1}{4\epsilon^2}} |Q_{r_j,\alpha_j}^{m}|
\\
&=& e^{-\frac{1}{4\epsilon^2}} |\{ x_m \in G: S_m^{*} g(x) > \lambda\}|.
\end{eqnarray*}
It completes the proof.
\end{proof}
For $1 \leq j \leq n$ we set $M^{j}$ be the Hardy-Littlewood maximal function on the space $G_j$ acting on functions defined on $G$. We denote by $M$ the strongly maximal function on the product space $G_1 \times \cdots \times G_n$,  that is, $M = M^{1}\circ \cdots \circ M^{n}$.
\

For $q>1$, we let $M_q^{j} (f) = (M^{j} (f^q))^{1/q}$ and $M_q (f) := (M_q (f^q))^{1/q}$. Set $\mathcal{M}_q = M_q \circ M_q \circ M_q$ and
\begin{eqnarray*}
G_r f(x) = \Biggl( \sum_{k_1,\cdots,k_n} \left| \mathcal{M}_q\left(\psi_{k_1,\cdots,k_n} (L_1^{\sharp},\cdots,L_n^{\sharp})f\right)(x) \right|^2 \Biggr)^{1/2}.
\end{eqnarray*}
\begin{lem} We have
\begin{eqnarray*}
\mathbb{D}_{k_1,\cdots, k_n} \left( \psi_{l_1,\cdots,l_n} (L_1^{\sharp},\cdots, L_s^{\sharp}) f(x) \right) \leq 2^{- \frac{1}{n q'} \sum_{j=1}^{n} \left| \frac{l_j}{2} + \log (\delta_j) k_j\right|} M_q f(x).
\end{eqnarray*}
\end{lem}
\begin{proof}
From Lemma 3.3, for each $1 \leq j \leq n$, there exists $a_j >0$ such that
\begin{eqnarray*}
\mathbb{D}_{k_j}^{j} \left( \psi_{l_1,\cdots,l_n} (L_1^{\sharp},\cdots,L_n^{\sharp}) f \right) (x) &\lesssim& 2^{-\left|\frac{l_j}{2} + \log(\delta_j) k_j \right|a_j }M_q^{j} f(x)
\\
&\lesssim& 2^{-\left| \frac{l_j}{2} + \log(\delta_j)k_j\right|a_j} M_q f(x),
\end{eqnarray*}
where the second inequality comes from  $M_q^{j} f(x) \leq M_q f(x)$.
Therefore we have
\begin{eqnarray*}
\mathbb{D}_{k_1,\cdots,k_n} \left( \psi_{l_1,\cdots,l_n} (L_1^{\sharp},\cdots, L_n^{\sharp}) f\right)(x) \lesssim 2^{- \left| \frac{l_j}{2} + \log(\delta_j) k_j \right|a_j} M_q f(x)
\end{eqnarray*}
for all $1 \leq j \leq n$. Set $a=\frac{1}{n} \min_{1\leq j\leq n}\{a_j\}$. If we product the above inequalities for all $1 \leq j \leq n$, we get 
\begin{eqnarray*}
\mathbb{D}_{k_1,\cdots,k_n} \left( \psi_{l_1,\cdots,l_n} (L_1^{\sharp},\cdots, L_n^{\sharp}) f\right)(x) \lesssim 2^{-  \sum_{j=1}^{n}\left| \frac{l_j}{2} + \log(\delta_j) k_j \right|\cdot a} M_q f(x).
\end{eqnarray*}
It proves the Lemma.
\end{proof}

\begin{lem} We have
\begin{eqnarray}
\psi_{l_1,\cdots,l_n} (L_1^{\sharp},\cdots, L_n^{\sharp}) m(L) f (x) \lesssim M f(x).
\end{eqnarray}
\end{lem}
\begin{proof}
If $m(\xi_1,\cdots,\xi_n) = m_1 (\xi_1)\cdots, m_n (\xi_n)$, then the lemma follows by using Lemma 2.4 repeatedly.
\

 In the general case, we write $m$ in the fourier series, 
\begin{eqnarray*}
m(\xi_1,\cdots,\xi_n) &=& \sum_{c_i \in \mathbb{Z}} e^{ic_1 \xi_1}\cdots e^{ic_n \xi_n} a_{c_1,\cdots,c_n} \psi(\xi_1)\cdots \psi(\xi_n).
\end{eqnarray*}
If we impose a sufficient regularity on $m$, the coefficents $a_{c_1,\cdots,c_n}$ decrease rapidly. Then we can use the above criterion to finish the proof of the lemma.\end{proof}

\begin{lem}
$G_n (f) (x) \geq C S_{n+1} (m(L) f) (x).$ 
\end{lem}
\begin{proof}
Recall that
\begin{eqnarray*}
S_{n+1} (m(L) f) (x) = \Biggl( \sum_{k_1,\cdots,k_n} (D_{k_1,\cdots,k_n} (m(L) f))(x)^2 \Biggr)^{1/2}.
\end{eqnarray*}
We have
\begin{eqnarray*}
\left|D_{k_1,\cdots,k_n} (m(L) f) (x)\right| &=&\Biggl| D_{k_1,\cdots,k_n} ( \sum_{l_1,\cdots,l_n} \psi_{l_1,\cdots,l_n} (L_1^{\sharp},\cdots,L_n^{\sharp})^3 m(L) f ) (x)\Biggr|
\\
&=& \Biggl|\sum_{l_1,\cdots,l_n} D_{k_1,\cdots,k_n} (\psi_{l_1,\cdots,l_n}(L_1^{\sharp},\cdots,L_n^{\sharp}))^2 m(L) \psi_{L_1,\cdots,l_n} (L_1^{\sharp},\cdots,L_n^{\sharp}) f ) (x)\Biggr|
\\
&\lesssim& \sum_{l_1,\cdots,l_n} 2^{-a \sum_{j=1}^{n} |\frac{l_j}{2} + \log (\delta_j) k_j|} M_q (M_q (\psi_{l_1,\cdots,l_n} (L_1^{\sharp},\cdots,L_n^{\sharp}) f )) (x).
\\
&\lesssim& \Biggl( \sum_{l_1,\cdots,l_n} 2^{-a \sum_{j=1}^{n} \left| \frac{l_j}{2} + \log (\delta_j) k_j \right| } \mathcal{M}_q(\psi_{l_1,\cdots,l_n}(L_1^{\sharp},\cdots,L_n^{\sharp}f )^2 (x) \Biggr)^{1/2}.
\end{eqnarray*}
Using this we get,
\begin{equation}
\begin{split}
&\left( S_{n+1} (m(L)f)(x)\right)^2
\\
&\quad\quad\quad\quad\quad\quad\quad = \sum_{k_1,\cdots,k_n} |D_{k_1,\cdots,k_n} (m(L)f)(x) |^2 
\\
&\quad\quad\quad\quad\quad\quad\quad \lesssim \sum_{k_1,\cdots,k_n} \sum_{l_1,\cdots,l_n} 2^{- a \sum_{j=1}^{n} \left| \frac{l_j}{2} + \log (\delta_j)k_j \right|}\mathcal{M}_q (\psi_{l_1,\cdots,l_n} (L_1^{\sharp},\cdots,L_n^{\sharp}) f )^2 (x)
\\
&\quad\quad\quad\quad\quad\quad\quad \lesssim \sum_{l_1,\cdots,l_n} \mathcal{M}_q (\psi_{l_1,\cdots,l_n} (L_1^{\sharp},\cdots,L_n^{\sharp}) f )^2 (x)
\\
&\quad\quad\quad\quad\quad\quad\quad  G_n (f)^2 (x),
\end{split}
\end{equation}
which is the asserted estimate.
\end{proof}

\begin{proof}[Proof of Theorem \ref{thm3}]
Set 
\begin{eqnarray*}
T_i^{1}(f) = \mathcal{A}( m_i (f))
\end{eqnarray*}
and
\begin{eqnarray*}
T_i^2(f) = m_i (f) - \mathcal{A}( m_i (f)) = \left( 1- (1-\mathcal{E}_1)\cdots (1-\mathcal{E}_n)\right) (m_i (f)).
\end{eqnarray*}
Then $m_i (f) = T_i^1 (f) + T_i^2 (f)$ and,
\begin{eqnarray*}
\sup_i |m_i (f) (x) | \leq \sup_i |T_i^1 (f)(x)| + \sup_{i} |T_i^2 (f)(x)|.
\end{eqnarray*}
Wee see that
\begin{equation*}
\{ x : \sup_i |m_i (f)(x)| > 2\lambda\} \subset \{ x : \sup_i |T_i^1 (f)(x)|> \lambda\} \cup \{ x : \sup_i |T_i^2 (f)(x)|> \lambda \}.
\end{equation*}
As for $T_i^{2}(f)$, we use Lemma 3.3 to get $\mathcal{E}_j (m_i (f))(x) \lesssim 2^{-N}M f(x)$. Using this and the trivial bound $\mathcal{E}_l (f)(x) \leq M f(x)$ for any $l$, we deduce that  $T_i^2 (f)(x) \lesssim 2^{-N} \mathcal{M} f(x),$
where $\mathcal{M} = M\circ \cdots \circ M$. Thus we get
\begin{eqnarray*}
\| T_i^{2}(f)(x)\|_{L^p} \lesssim 2^{-N} \| \mathcal{M} f\|_{L^p} \lesssim 2^{-N} \| f\|_{L^p}.
\end{eqnarray*}
 For the main term $T_i^{1}(f)$, we set $A_{\lambda} := \{ x \in G : \sup_{1\leq i \leq n} |T_i^{1} (f)(x)| > \lambda \}$. Then, 
\begin{equation*}
A_{\lambda}  \subset \{ x : \sup_i |T_i^1 (f)| > \lambda,~ G_r (f) (x) \leq C\epsilon^n \lambda \} \cup \{ x : G_r (f) (x) > C\epsilon^n \lambda \},
\end{equation*}

Since $G_n (f) (x) \geq C S_{n+1} (T_i^{1} f)(x)$ we have
\begin{equation*}
\{ x: \sup_i |T_i^{1}(f)(x)|> \lambda,~ G_r (f)(x) \leq C \epsilon^n \lambda\} \subset \bigcup_{i=1}^{N} B_{i,\lambda}
\end{equation*}
where $B_{i,\lambda} =  \{ x : |T_i^{1} (f)(x)| > \lambda, S_{n+1} (T_i^{1}f)(x) \leq  \epsilon^n \lambda \}.$ Then, since 
\begin{eqnarray*}
|A_{\lambda} | &\leq & \left| \bigcup_{i=1}^{N} B_{i,\lambda} \cup \{ x :G_r f(x) > C \epsilon^{n} \lambda\} \right|
\\
&\leq & \sum_{i=1}^{N} |B_{i,\lambda}| + |\{ x : G_r f(x) > C \epsilon^n \lambda\}|,
\end{eqnarray*} 
we have
\begin{equation}\label{sup1i}
\begin{split}
&\left\| \sup_{1 \leq i \leq N} | T_i^{1} (f)(x) | \right\|_{L^p (G)}^p
\\
&\quad\quad \leq \int p \lambda^{p-1} |A_{\lambda}| d\lambda 
\\
&\quad\quad \leq \sum_{i=1}^{N} \int_0^{\infty} p \lambda^{p-1} |B_{i,\lambda}| d\lambda + \int_0^{\infty} p \lambda^{p-1} | \{ x : G_r (f)(x) > C \epsilon^n \lambda \}| d\lambda. 
\end{split}
\end{equation}
For each $1 \leq i \leq n$, we split the set $B_{i,\lambda}$  as follows.\begin{equation*}
\begin{split}
B_{i,\lambda} \subset \{ x : |T_i^1 (f)(x)| > \lambda,~ S_{2} (T_i^1 f) (x) \leq \epsilon \lambda \} \cup \{ S_2 (T_i^{1} f)(x) >\epsilon \lambda,~S_{n+1} (T_i^1 f) (x) \leq \epsilon^n \lambda \}.
\end{split}
\end{equation*}
Similarly, for $1 \leq k \leq n-1$ we have
\begin{equation*}
\begin{split}
\{S_{k+1} (T_i^1 f) (x) > \epsilon^k \lambda, ~ S_{n+1} (T_i^{1} f)(x) \leq \epsilon^n \lambda \} &
\\
\subset \{S_{k+1} (T_i^1 f ) (x)  > \epsilon^k \lambda,~ &S_{k+2} (T_i^1 f) (x) < \epsilon^{k+1} \lambda \}
\\
 \cup \{ S_{k+2}& (T_i^1 f) (x) > \epsilon^{k+1}\lambda ,~ S_{n+1}(T_i^{1}f)(x) \leq \epsilon^n \lambda \}.
\end{split}
\end{equation*}
For $k=n-1$ the last set in the above equation is empty. Therefore,
\begin{equation*}
\begin{split}
B_{i,\lambda} \subset \{ x : |T_i^{1}(f)(x)|>\lambda, ~ S_2 (T_i^{1} f) (x) < &~\epsilon \lambda \} 
\\
 \bigcup_{i=2}^{n} \{x : S_k (T_i^{1} f)(x) &> \epsilon^{k-1} \lambda, ~S_{k+1}(T_i^{1} f)(x) < \epsilon^{k} \lambda \}.
\end{split}
\end{equation*}
Using Lemma \ref{for2m} we deduce that
\begin{eqnarray*}
|B_{i,\lambda}| &\leq& |\{ x : |T_i^{1}(f)(x)| > \lambda, ~S_2(m_i^{1} f)(x) < \epsilon \lambda\}|
\\
& &+ \sum_{k=2}^{n} | \{ x : S_k (T_i^{1} f)(x) > \epsilon^{k-1} \lambda, ~S_{k+1}(T_i^{1} f)(x) < \epsilon^{k}\lambda \}|
\\
&\leq & \sum_{k=1}^{n} e^{-\frac{C}{\epsilon^2}} \left| \left\{ x : |S_k^{*}(T_i^{1}f)(x) \geq \frac{1}{2}\epsilon^{k-1}\lambda \right\} \right|.
\end{eqnarray*}
Using this we can bound \eqref{sup1i} as follows.
\begin{equation}
\begin{split}
&\left\| \sup_{1 \leq i \leq N} | T_i^{1} (f)(x) | \right\|_{L^p (G)}^p
\\
&\quad\quad \lesssim   \sum_{i=1}^{N} \sum_{k=1}^{n} e^{-\frac{C}{\epsilon^2}} \int \lambda^{p-1} |\{ x : |S_k^{*} (T_i^{1} f) (x) \geq \frac{1}{2} \epsilon^{k-1}\lambda \} | d \lambda+ \epsilon^{-np} \| G_r (f) (x) \|_p^{p}
\\
&\quad\quad\lesssim  \sum_{i=1}^{N} \sum_{k=1}^{n} e^{-\frac{C}{\epsilon^2}} \int \epsilon^{-(k-1)p} \lambda^{p-1} | \{ x : |S_k^{*} (T_i^{1} f) (x) \geq \frac{1}{2} \lambda \} | d\lambda   + \epsilon^{-np} \| G_r (f) (x) \|_p^{p}
\\
&\quad\quad\lesssim  N e^{-\frac{C}{\epsilon^2}}\epsilon^{-(n-1)p} \sup_{k, i} \| S_k^{*} (T_i^{1} f)\|_{L^p}^p +  \epsilon^{-np} \| G_r (f) (x) \|_p^{p}.
\end{split}
\end{equation}
Take $ \epsilon = (\log {N+1})^{1/2}$, then from $\sup_{k,i} \| S_k^{*} (T_i^1 f) \|_{L^p} \lesssim \| f\|_{L^p}$ and $\|G_r (f) (x) \|_p \lesssim \| f\|_p$ ,  we get,
\begin{eqnarray*}
\|\sup_{1 \leq i \leq N} | T_i^{1} (f) (x)| \|_{L^p} \lesssim (\log{N+1})^{n/2} \|f \|_{L^p}.
\end{eqnarray*} 
This yields the desired inequality.
\end{proof}
For the joint spectral multiplier, the $L^p$ boundedness was proved by M\"uller-Ricci-Stein (see \cite[Lemma 2.1]{MRS}). They made use of the transference argument of Coifmann-Weiss \cite{CW} to prove Theorem \ref{thm3}.
\begin{proof}[Proof of Theorem \ref{thm4}]
Set $G=\mathbb{H}_n \times \mathbb{R}$. For a function $f$ defined on $G$, we set a function  $f^{b}$ defined on $\mathbb{H}_n$, 
\begin{eqnarray*}
f^{b} (z,t) = \int^{\infty}_{-\infty} f (z,t-u,u) du.
\end{eqnarray*}
Let $K$ be the kernel of a multiplier $m(L^{\sharp},iT).$ Then, $K^{b}$ equals to the kernel of $m(L,iT)$ (see \cite[p. 207]{MRS}). Thus,
\begin{eqnarray*}
 m(L, iT) \phi (z,t) &=&\phi * K^{b} (z,t)
\\
&=& \int_{\mathbb{H}_n} \phi ((z,t) \cdot (z',w)^{-1}) \left[ \int_{\mathbb{R}} K(z', w-u',u') du' \right] dz' dw
\\
&=&\int_G K(z,',t',u') \phi ((z,t) \cdot (z', t' + u')^{-1}) dz' dt' du'
\end{eqnarray*}
We temporarily suppose that support of $K_j (z',t',u')$ in $u'$ variable is in $[-M,M]$ for a fixed $M>0$.  For each $R \in \mathbb{N}$, we set  $\chi_R $  be the characteristic function on $[-2R, 2R]$. If $R \geq 10M$, then we get the following.
\begin{equation*}
\begin{split}
\bigl\| \sup_{1 \leq j \leq N} \bigr.&\bigl. m_j (L, iT) \phi (z,t) \bigr\|_{L^p (z,t)}  
\\
=& ~\Bigl\| \sup_{1 \leq j \leq N} \Bigl|\int_G K_j (z', t' ,u')( \phi ((z,t) \cdot (-z', -(t'+u')) dz' dt' du' \Bigr|\Bigr\|_{L^p (z,t)}
\\
\leq &~\frac{1}{R^{1/p}} \Bigl\| \sup_{1 \leq j \leq N} \left|\int_G (K_j (z,',t',u')) (\chi_R(u-u') \phi ((z,t+u) \cdot (-z',-(t' + u')) ) dz' dt' du' ) \Bigr|\right\|_{L^p (z,t,u)}
\\
\leq& ~\frac{1}{R^{1/p}} \Bigl\| \sup_{1 \leq j \leq N} | m_j (L^{\sharp}, iT)| \Bigr\|_{L^p \rightarrow L^p} \left\| \chi_R (u) \phi(z, t+u) \right\|_{L^p (z,t,u)}
\\
\leq& ~ 10~ \Bigl\| \sup_{1 \leq j \leq N}|m_j (L^{\sharp},iT)| \Bigr\|_{L^p \rightarrow L^p } \|\phi \|_{L^p (\mathbb{H}_n)}.
\end{split}
\end{equation*}
Since the above estimate does not depend on $M$, we can use an approximation argument to delete the assumption on the suppor of $K_j$. It gives the asserted inequality.
\end{proof}

\section*{Acknowledgements}\thispagestyle{empty} 
I am thanksful to my advisor Rapha\"el Ponge for his support and careful proofreading during the preparation of this paper, as well as for his introducing the paper \cite{S} which motivated me to learn the background for this paper.

\end{document}